\documentclass[a4paper,12pt]{article}
\usepackage[top=2.5cm,bottom=2.5cm,left=2.5cm,right=2.5cm]{geometry}
\usepackage{cite,amsmath,amssymb}
    \usepackage[margin=1cm,%
                font=small,%
                format=hang,%
                labelsep=period,%
                labelfont=bf]{caption}
\usepackage{pgf,subfigure}
\usepackage{amsfonts,epsf,here,graphicx}
\usepackage{latexsym,multirow,rotating}

\usepackage{epstopdf}

\newtheorem{theorem}{\bf Theorem}[section]
\newtheorem{corollary}[theorem]{\bf Corollary}

\newtheorem{proposition}[theorem]{\bf Proposition}

\newcommand{\qed}{\hfill $\square$ \bigskip}

\newcommand{\aut}{\rm Aut}
\newcommand{\gp}{\rm GP}

\begin{document}

\baselineskip=0.30in
\vspace*{40mm}

\begin{center}
{\LARGE \bf The Graovac-Pisanski index of  connected bipartite graphs with applications to hydrocarbon molecules}
\bigskip \bigskip

{\large \bf Matev\v z \v Crepnjak$^{a,b}$, Martin Knor$^d$, Niko Tratnik$^{a}$,\\ Petra \v Zigert Pleter\v sek$^{a,b}$}

\smallskip
{\em  $^a$Faculty of Natural Sciences and Mathematics, University of Maribor,  Slovenia} \\
{\tt matevz.crepnjak@um.si, \tt niko.tratnik@um.si, \tt petra.zigert@um.si}
\medskip

{\em  $^b$Faculty of Chemistry and Chemical Engineering, University of Maribor, Slovenia} \\
\medskip

{\em $^d$ Faculty of Civil Engineering, Slovak University of Technology in
Bratislava, Bratislava, Slovakia} \\
{\tt knor@math.sk}

\bigskip\medskip

(Received March 20, 2021)

\end{center}

\noindent
\begin{center} {\bf Abstract} \end{center}

\vspace{3mm}\noindent

The Graovac-Pisanski index, also called the modified Wiener index,  was introduced in 1991 and represents an extension  of the original Wiener index, because it considers beside the distances in a graph also its  symmetries.  Similarly as Wiener in 1947 showed the correlation of the Wiener indices of the alkane series  with the  boiling points,  in 2018  the  connection between the  Graovac-Pisanski index and the  melting points of some hydrocarbon molecules was established.  
In this paper, we prove that the Graovac-Pisanski index of any connected bipartite graph as well as  of any connected graph on an even number of vertices is an integer number. These results are applied to some important families of hydrocarbon molecules.
By using a computer programme, the graphs with a non-integer Graovac-Pisanski index  on at most nine vertices are counted. Finally, an infinite class of unicyclic graphs with a non-integer Graovac-Pisanski index is described.


\baselineskip=0.30in


 \medskip\noindent
 {\bf AMS Subj. Class:} 05C92, 05C09, 92E10, 05C12.


\section{Introduction}

In \cite{todes} we can find the following definition:
"The molecular descriptor is the final result of a logic and mathematical procedure which transforms chemical information encoded within a symbolic representation of a molecule into a useful number or the result of some standardized experiment."  

Since the seminal research by Wiener from 1947 \cite{Wiener}   many molecular descriptors, i.e.\,graph invariants, were introduced. Different variations of  the Wiener index  have been known  since then, one of them being the  Graovac-Pisanski index \cite{graovac}, also called the  modified Wiener index. The advantage of this index in comparison to other distance-based indices is in the encountering of the symmetries of a graph, since the symmetries of a molecule have an influence on its properties \cite{pinal}. 

On one side, the  mathematical properties of  the  Graovac-Pisanski index were investigated. For example, the index was intensely studied for nanostructures \cite{ashrafi_diu,ashrafi_koo_diu,hakimi,koo_ashrafi3,koo_ashrafi,sha_ashrafi}, some general results on the Graovac-Pisanski index were obtained \cite{ashrafi_sha,ghorbani}, the closed formulae for carbon nanotubes were calculated \cite{tratnik,trat-plet}, extremal trees were considered \cite{knor2019}, and the symmetries of molecules were studied \cite{koo_ashrafi2,koo_ashrafi4}. On the other side, 
the chemical usefulness of the Graovac-Pisanski index was shown in \cite{ashrafi_koo_diu1}, where the connection  with the topological efficiency was considered  and later in \cite{crepnjak}, where the correlation with the melting points of some families of hydrocarbon molecules was established. This was done with the use of the QSPR analysis and for all the considered molecular graphs from \cite{crepnjak} the Graovac-Pisanski index is an integer number. Naturally, the question whether  the  Graovac-Pisanki index is an integer number for all connected graphs arises by this observation.

We proceed as follows. In the next section, some basic notation is introduced and important definitions are included. In Section \ref{main}, we prove that the Graovac-Pisanski index of any connected bipartite graph is an integer number. Moreover, we show that the same conclusion holds also for any connected graph with an even number of vertices. In particular,  this means that for various molecular graphs the Graovac-Pisanski index is an integer, what is explained in  Section \ref{appl}. In the same section, the closed formula for the Graovac-Pisanski index of an infinite family of nanotubical fullerenes is deduced.  Finally, in Section \ref{non_in} we use a computer programme to obtain the number of connected graphs on at most nine vertices for which the Graovac-Pisanski index is not an integer. An infinite family of graphs with a non-integer Graovac-Pisanski index is also included.

\section{Preliminaries}

All the graphs considered in this paper are simple, finite
and connected.
The {\em distance} $d_G(u,v)$ between vertices $u$ and $v$ of a graph $G$
is the length of a shortest path between vertices $u$ and $v$ in $G$.
 A \textit{bipartite graph} is a graph whose vertices can be divided into two disjoint sets $U$ and $V$ such that every edge connects a vertex in $U$ to one in $V$. Vertex subsets $U$ and $V$ will be called the \textit{bipartite sets} of the graph.

The {\em Wiener index} of a connected graph $G$ is defined as follows:
$$
W(G) = \sum_{\lbrace u,v \rbrace \subseteq V(G)} d_G(u,v).
$$
Also, for any $S \subseteq V(G)$ we define $$W_G(S) = \sum_{\lbrace u,v \rbrace \subseteq S} d_G(u,v).$$

An \textit{isomorphism of graphs} $G$ and $H$ 
is a bijection $f: V(G)\to V(H)$, such that for any two vertices $u$ and $v$ from $G$
it holds that $u$ and $v$ are adjacent in $G$ if and only if $f(u)$
and $f(v)$ are adjacent in $H$.
If $f: V(G)\to V(G)$ is an isomorphism then the function $f$ is called
an \textit{automorphism} of the graph $G$.
It is easy to check that the composition of two automorphisms is another automorphism. Moreover, the set of all automorphisms of a given graph $G$, under
the composition operation, forms the \textit{automorphism group} of $G$, denoted by ${\aut}(G)$.

Let $G$ be a connected graph. The \textit{Graovac-Pisanski index} of $G$
is defined as
$$
{\gp}(G) = \frac{|V(G)|}{2 |{\aut}(G)|} \sum_{u \in V(G)} \sum_{\alpha \in {\aut}(G)} d_G(u, \alpha(u)).
$$
Let us mention, that for the Graovac-Pisanski index, also known as the modified Wiener index, the notation $\widehat{W}(G)$ is used, but to clear the confusion in the area of variations of the Wiener indices, we use the name Graovac-Pisanski index, as suggested in \cite{ghorbani} and denote it with the ${\gp}(G)$.

Finally, we include some basic definitions from group theory.
If $G$ is a group and $X$ is a set, then a \textit{group action}
$\phi$ of $G$ on $X$ is a function $\phi :G \times X \to X$
that satisfies the following: $\phi(e,x) = x$ for any $x \in X$
(here, $e$ is the neutral element of $G$) and $\phi(gh,x)=\phi(g,\phi(h,x))$
for all $g,h \in G$ and $x \in X$.
The \textit{orbit} of an element $x$ in $X$ is the set of elements
in $X$ to which $x$ can be moved by the elements of $G$,
i.e. the set $\lbrace \phi(g,x) \, | \, g \in G \rbrace$.
If $G$ is a graph and ${\aut}(G)$ the automorphism group,
then $\phi: {\aut}(G) \times V(G) \to V(G)$, defined by
$\phi(\alpha,u) = \alpha(u)$ for any $\alpha \in {\aut}(G)$,
$u \in V(G)$, is called the \textit{natural action} of the group
${\aut}(G)$ on $V(G)$.

It was shown by Graovac and Pisanski \cite{graovac} that if $V_1, \ldots, V_t$ are the orbits
under the natural action of the group ${\aut}(G)$ on $V(G)$, then
\begin{equation}
\label{eq:GP2}
{\gp}(G) = |V(G)| \sum_{i=1}^t \frac{1}{|V_i|}W_G(V_i).
\end{equation}

\section{Main results}
\label{main}

In this section we prove the main result of the paper, i.e.\ the Graovac-Pisanski index of any connected bipartite graph is always an integer number.

First, we will rewrite the Graovac-Pisanski index in another form. For this purpose, we need some additional definitions. Let $G$ be a graph and let $u\in V(G)$.
The {\em distance of} $u$ in $G$, $w_G(u)$, is the sum of  from $u$
to all the other vertices of $G$.
That is, $w_G(u)=\sum_{v\in V(G)}d_G(u,v)$. Moreover, if $S \subseteq V(G)$ and $u \in V(G)$, then $w_S(u)=\sum_{v\in S}d_G(u,v)$.
Using the distance, one can rewrite the Wiener index for $S \subseteq V(G)$ as follows
\begin{equation}
\label{eq:W}
W_G(S)=\frac 12\sum_{u\in S}w_S(u).
\end{equation}

\begin{proposition} \label{propo3}
Let $G$ be a connected graph and let $v_1,\dots,v_t$ be the representatives of orbits $V_1,\dots,V_t$ under the natural action of the group ${\aut}(G)$ on $V(G)$. Then
\begin{equation}
\label{eq:GP3}
{\gp}(G) = |V(G)| \sum_{i=1}^t \frac 12 w_{V_i}(v_i),
\end{equation}
\end{proposition}

\begin{proof}
If two vertices $u,v$ are in the same orbit $V_i$, $i \in \lbrace 1, \ldots, t \rbrace$, then it obviously holds $w_{V_i}(u) = w_{V_i}(v)$. Since all the vertices in the same orbit have the very same sum of distances, inserting (\ref{eq:W}) into (\ref{eq:GP2}) we get 
$${\gp}(G) = |V(G)| \sum_{i=1}^t \frac{|V_i| \cdot w_{V_i}(v_i)}{2|V_i|} = |V(G)| \sum_{i=1}^t \frac 12 w_{V_i}(v_i)$$

\noindent
and the required result follows. \qed
\end{proof}


\noindent
By (\ref{eq:GP3}) we get the following observations.

\begin{corollary} \label{sodo}
If $G$ is a connected graph with an even number of vertices, then  ${\gp}(G)$ is an integer
number.
\end{corollary}

\begin{corollary}
\label{prop:half}
The Graovac-Pisanski index of any connected graph is either an integer
or half of an integer number.
\end{corollary}

\noindent
However, if $G$ is bipartite, we can say more. The following result is the main result of this paper.

\begin{theorem}
\label{thm:bipartite}
If $G$ is a connected bipartite graph, then ${\gp}(G)$ is an integer
number.
\end{theorem}

\begin{proof}
Let $G$ be a bipartite graph and let $U,V$ be the bipartition of $V(G)$.
That is, no edge connects two vertices of $U$  and no edge connects two
vertices of $V$. We consider two cases:
\begin{itemize}
\item [$(i)$] \textit{There is $\psi \in \aut(G)$ such that $\psi(u)=v$ for some $u \in U$ and $v \in V$.}\\
In this case, we will show that $\psi$ reverses the bipartite sets, i.e.\ $\psi(x) \in V$ for any $x \in U$ and $\psi(y) \in U$ for any $y \in V$. To prove this, suppose that $x \in U$ but $\psi(x)$ also belongs to $U$. Since $u,x \in U$, the distance $d_G(u,x)$ is an even number. Therefore, since $\psi$ is an automorphism, $d_G(\psi(u),\psi(x))=d_G(u,x)$ must be even, too. Since $\psi(u)=v$, we obtain that $d_G(v,\psi(x))$ is even. But $\psi(x) \in U$ and $v \in V$ and therefore, this is a contradiction. In a similar way one can show $\psi(y) \in U$ for any $y \in V$. We have proved that $\psi$ reverses the bipartite sets. Moreover, the restriction of $\psi$ on $U$, $\psi|_U$, is a bijection from $U$ to $V$. Consequently, $|U|=|V|$ and $|V(G)|$ is even. By Corollary \ref{sodo}, the Graovac-Pisanski index of $G$ is an integer number.


\item [$(ii)$] \textit{There is no automorphism that reverses the bipartite sets.} \\
In this case, if $u$ and $v$ are two vertices from the same orbit, these two vertices are also in the same bipartite set. Therefore, the distance between them must be even. Consequently, if $v_1,\dots,v_t$ are the representatives of orbits $V_1,\dots,V_t$ under the natural action of the group ${\aut}(G)$ on $V(G)$, all the distances $w_{V_i}(v_i)$ are even in (\ref{eq:GP3}), and
so the Graovac-Pisanski index is again an integer number.
\end{itemize}
Since the Graovac-Pisanski index is an integer number in both cases, the proof is complete. \qed
\end{proof}

\section{Applications to molecular graphs}
\label{appl}

In this section we apply our main results to some well-known molecular graphs.

\subsection{Paraffins, benzenoid hydrocarbons, phenylenes and carbon nanotubes}

{\it Paraffins} (alkanes) are mathematically modelled by chemical trees, i.e.\,trees in which every vertex has degree at most four. 

Let ${\cal H}$ be the hexagonal (graphite) lattice and let $Z$ be a cricuit on it. Then a {\it benzenoid system} is induced by the vertices and edges of ${\cal H}$, lying on $Z$ and in its interior. In chemistry a benzenoid system ia a mathematical model for a benzenoid hydrocarbon. Let $G$ be a benzenoid system. A vertex shared by three hexagons of $G$ is called an \textit{internal} vertex of $G$. A benzenoid system is said to be \textit{catacondensed} if it does not possess internal vertices. Otherwise it is called \textit{pericondensed}. Two distinct hexagons with a common edge are called \textit{adjacent}.  Let $G$ be a catacondensed benzenoid system. If we add squares between all pairs of adjacent hexagons of $G$, the obtained graph $G'$ is called a \textit{phenylene}, see Figure \ref{benphen}.

\begin{figure}[!htb]
	\centering
		\includegraphics[scale=0.6, trim=0cm 0cm 1cm 0cm]{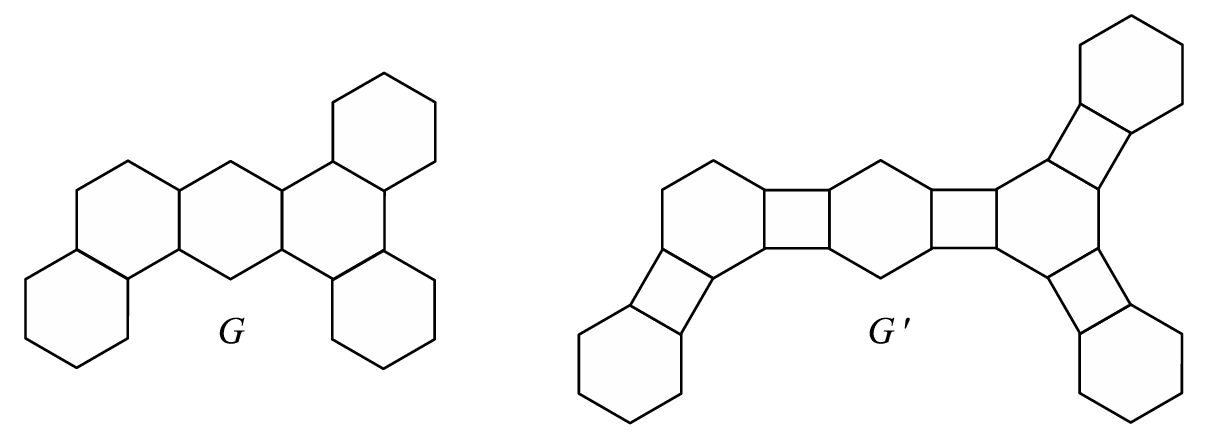}
\caption{A benzenoid system $G$ and corresponding phenylene $G'$.}
\label{benphen}
\end{figure}

Next we  formally define open-ended carbon nanotubes, also called tubulenes. Choose any lattice point in the hexagonal lattice as the origin $O$. Let $\overrightarrow{a_1}$ and $\overrightarrow{a_2}$ be the two basic lattice vectors.
 Choose a vector $ \overrightarrow{OA} =n\overrightarrow{a_1}+m \overrightarrow{a_2}$
such that $n$ and $m$ are two integers and $|n|+|m|>1$, $nm\neq -1$. Draw two straight lines $L_1$ and $L_2$ passing through
$O$ and $A$ perpendicular to $O A$, respectively. By rolling up the hexagonal strip between $L_1$ and $L_2$ and gluing $L_1$ and $L_2$ such
that $A$ and $O$ superimpose, we can obtain a hexagonal tessellation $\mathcal{HT}$ of the cylinder. $L_1$ and $L_2$ indicate the direction of
the axis of the cylinder. Using the terminology of graph theory, a {\it tubulene} $G$ is defined to be the finite graph induced by all
the hexagons of $\mathcal{HT}$ that lie between $c_1$ and $c_2$, where $c_1$ and $c_2$ are two vertex-disjoint cycles of $\mathcal{HT}$ encircling the axis of
the cylinder.

For any  tubulene $G$, if its chiral vector is $ n \overrightarrow{a_1} + m \overrightarrow{a_2}$, $G$ will be called an $(n,m)$-type tubulene, see Figure \ref{zig-zag} \cite{tratnik}. If $G$ is a $(n,m)$-type tubulene where $n=0$ or $m=0$, we call it a \textit{zig-zag tubulene}.

\begin{figure}[!htb]
	\centering
		\includegraphics[scale=0.6, trim=0cm 0cm 1cm 0cm]{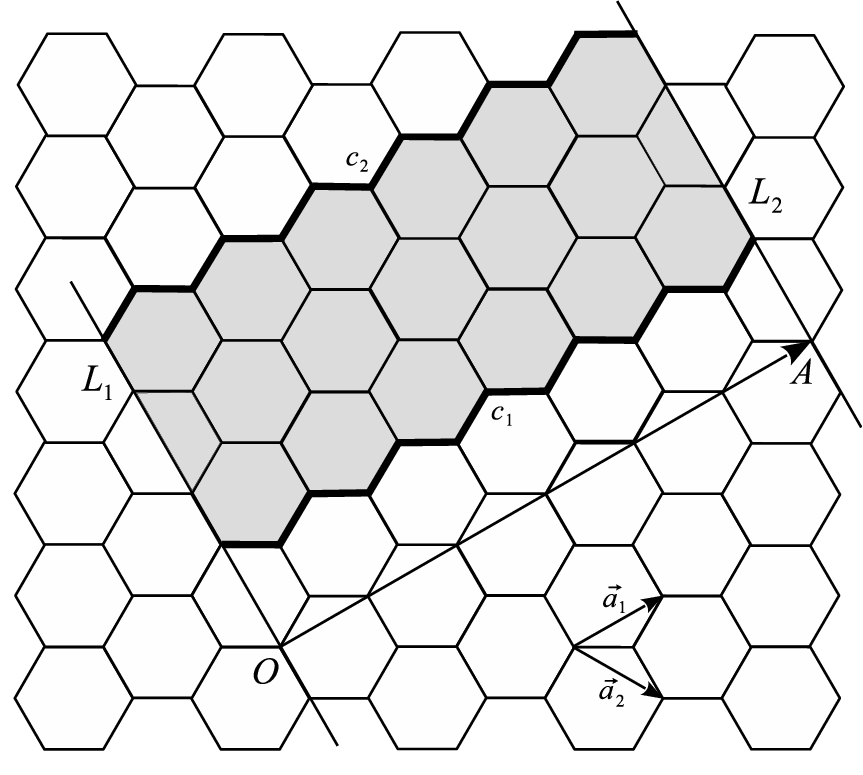}
\caption{A $(6,0)$-type tubulene (zig-zag tubulene).}
	\label{zig-zag}
\end{figure}

\begin{proposition}
 Let $G$ be a paraffin, a benzenoid system, a phenylene or a tubulene. The Graovac-Pisanski index of $G$ is an integer
number. 
\label{molecular}
\end{proposition}

\begin{proof}
Obviously every paraffin is a bipartite graphs and the same holds for benzenoid systems and phenylenes. It was shown in \cite{tr-zi} that tubulenes are
bipartite graphs as well. Therefore by Theorem \ref{thm:bipartite} the Graovac-Pisanski index of $G$ is an integer number.
\qed
\end{proof}

\subsection{Fullerenes}
A \textit{fullerene} $G$ is a  $3$-connected 3-regular plane graph such that every face is bounded by either a pentagon or a hexagon, see Figure \ref{grafF} (obviously, fullerenes
are not bipartite graphs).  By Euler's formula, it follows that 
the number of pentagonal faces of a fullerene is exactly $12$. For more 
information on fullerenes see \cite{andova}.

\bigskip

\begin{figure}[h!] 
\begin{center}
\includegraphics[scale=1.1]{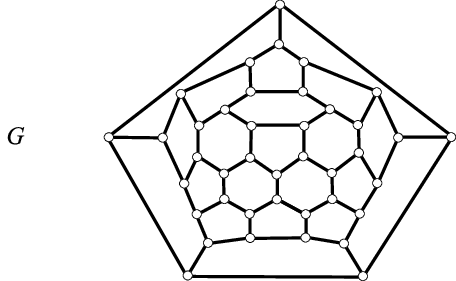}
\end{center}
\caption{\label{grafF} A fullerene $G$.}
\end{figure}

The following result about the Graovac-Pisanski index of fullerenes can now be stated. 

\begin{proposition} Let $G$ be a fullerene. The Graovac-Pisanski index of $G$ is an integer
number. \label{ful}
\end{proposition}

\begin{proof}
Since all vertices in a fullerene have odd degrees, the graph must have even
number of vertices.
(It is well known that a fullerene with $n$ vertices exists for any even
$n\geq 24$ and for $n=20$, for the details see Theorem 2.2 in \cite{andova}.)
Hence, by Corollary \ref{sodo} it follows that the Graovac-Pisanski index
of $G$ is an integer number. 
\qed
\end{proof}

To conclude the section, we compute the Graovac-Pisanski index of an infinite family of nanotubical fullerenes denoted by $F_n$, $n \geq 1$. In particular, $F_n$ is obtained by the $(6,0)$-type tubulene  with $n$ layers of hexagons (zig-zag tubulene such that every layer contains 6 hexagons), closed with two caps shown in Figure \ref{caps} \cite{andova1}. 

\begin{figure}[h!] 
\begin{center}
\includegraphics[scale=1.1, trim=0cm 0.3cm 0cm 0cm]{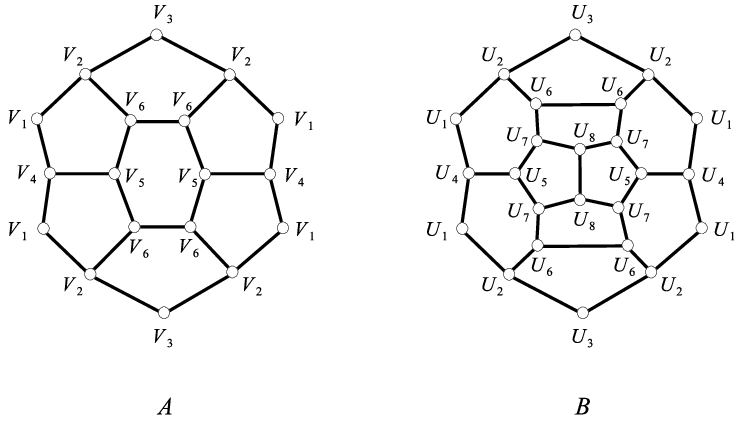}
\end{center}
\caption{\label{caps} Caps $A$ and $B$ for the fullerene $F_n$.}
\end{figure}

We will use Proposition \ref{propo3} to calculate the Graovac-Pisanski index of $F_n$. The orbits and corresponding representatives in the cap $A$ are denoted by $V_i$ and $v_i$, $i \in \lbrace 1,\ldots,6 \rbrace$. The orbits and corresponding representatives of the cap $B$ are denoted by $U_i$ and $u_i$, $i \in \lbrace 1,\ldots,8 \rbrace$.

First, we compute the distances of vertices $v_i$, $i \in \lbrace 1,\ldots,6 \rbrace$, in the corresponding orbits: $w_{V_1}(v_1)  =  2 + 4 + 6 = 12$,
$w_{V_2}(v_2)  =  2+ 4+ 5 = 11$,
$w_{V_3}(v_3)  =  6$,
$w_{V_4}(v_4)  =  5$,
$w_{V_5}(v_5)  =  3$,
$w_{V_6}(v_6)  =  1+2+3 = 6$.
The sum of all these distances is 43.

Moreover, we calculate the distances of vertices $u_i$, $i \in \lbrace 1,\ldots, 8 \rbrace$, in the corresponding orbits:
$w_{U_1}(u_1)  =  2 + 4 + 6 = 12$, 
$w_{U_2}(u_2)  =  2+ 4+ 6 = 12$,
$w_{U_3}(u_3)  =  6$,
$w_{U_4}(u_4)  =  6$,
$w_{U_5}(u_5)  =  4$,
$w_{U_6}(u_6)  =  1+4+5 = 10$,
$w_{U_7}(u_7)  =  2+2+3 = 7$,
$w_{U_8}(u_8)  =  1$.
The sum of all these distances is 58.
The orbits and the distances of the layers in the tubical part are analogous to the orbits $U_i$ and distances $w_{U_i}(u_i)$, $i \in \lbrace 1,2,3,4 \rbrace$.

Obviously, it holds $|V(F_n)| = 12(n+1) + 6 + 12 = 12n + 30$. Therefore, by Proposition \ref{propo3} we conclude
$$ {\gp}(F_n) = \frac{12n+30}{2}\left( 43 + 58 + 36(n-1) \right) = 216n^2 + 930n + 975.$$

We observe that the Graovac-Pisanski index of fullerene $F_n$ is an integer number for any $n\geq 1$. Also, the Graovac-Pisanski index of two other families of nanotubical fullerenes was computed in \cite{ashrafi_koo_diu1} and the result is again an integer number, what coincides with Proposition \ref{ful}.

\section{Graphs for which the Graovac-Pisanski index is not an integer}
\label{non_in}

By Theorem~{\ref{thm:bipartite}}, Corollary \ref{sodo}, and Proposition \ref{ful}, many chemical graphs
have integer Graovac-Pisanski index (for example the fullerene graphs, trees, hexagonal structures, etc.).
Of course, there are also connected graphs whose Graovac-Pisanski index is not
an integer.
The smallest such graph contains $5$ vertices. 

In Table \ref{tab:GP} we present the number of connected graphs on $n$ vertices, $n\in\{3,5,7,9\}$, whose
Graovac-Pisanski index is not an integer number. The results were obtained by using a computer programme.

\begin {table}[H]
\begin{center}
\begin{tabular}{|l|c|c|c|c|}
\hline
{Number of vertices} & 3 & 5 & 7 & 9 \\ 
\hline
{Number of connected graphs} & 2 & 21 & 853 & 261080 \\
\hline
{Number of graphs whose GP index is integer} & 2 & 14 & 516 & 197584 \\
\hline
{Number of graphs whose GP index is not integer} & 0 & 7 & 337 & 18496 \\
\hline
\end{tabular}
\caption{\label{tab:GP} 
Number of connected graphs on 3, 5, 7, and 9 vertices, divided in those with an integer and those with a non-integer   Graovac-Pisanski index }
\end{center}
\end {table}

\noindent
In Figure \ref{slika} we can see all the connected graphs on 5 vertices whose Graovac-Pisanski index is not an integer number.

\begin{figure}[h!] 
\begin{center}
\includegraphics[scale=0.8]{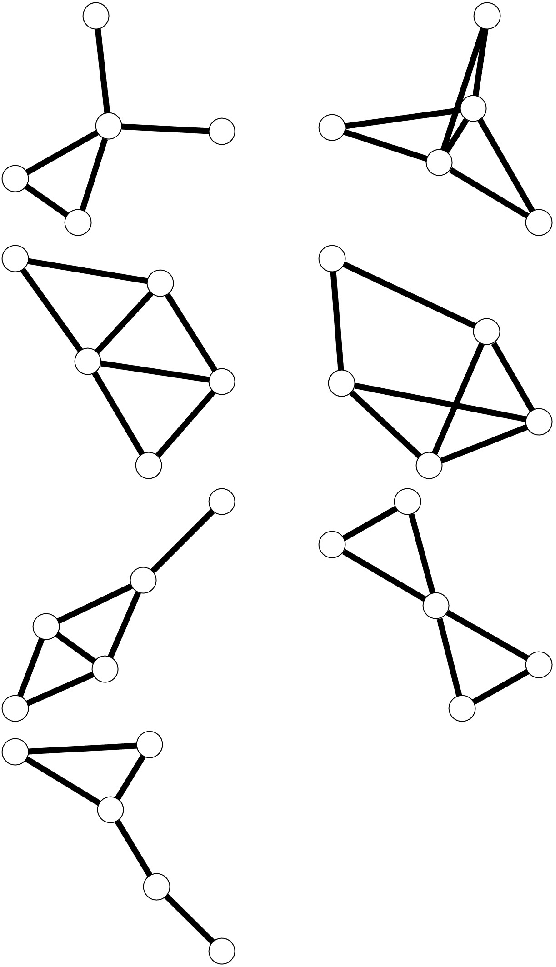}
\end{center}
\caption{\label{slika} All connected graphs on 5 vertices for which the Graovac-Pisanski index is not an integer.}
\end{figure}

The following theorem presents an infinite class of unicyclic graphs for which the Graovac-Pisanski
index is not an integer number.

\begin{theorem}
\label{thm:unicyclic}
Let $G$ be a graph consisting of a cycle of odd length $\ell$, where
$\ell\equiv 3\mbox{ or }5\pmod 8$, to which we attach a path of even length
$t$, where $t$ is a positive integer.
Then $G$ has $\ell+t$ vertices and ${\gp}(G)$ is not an integer number.
\end{theorem}

\begin{proof}
In $G$ there are $\frac{\ell-1}{2}$ orbits with two vertices and $t+1$ orbits
with one vertex.
Let $z$ be the vertex of degree $3$ in $G$.
Then the orbits with two vertices contain pairs of vertices at the same
distance from $z$.
Hence, the distances between the pairs of vertices from the same orbit are
$1,2,\dots,\frac{\ell-1}{2}$ and
$$
{\gp}(G)=\big|V(G)\big|\cdot\frac 12\cdot\Big(1+2+\dots+\tfrac{\ell-1}2\Big)
=\big|V(G)\big|\cdot\frac 12\cdot\binom{(\ell+1)/2}2.
$$
It is easy to observe that the number $|V(G)|=\ell + t$ is odd and consequently,
the Graovac-Pisanski index is not an integer if and only if the number representing 
the binomial coefficient, $\frac 12\frac{\ell+1}2\frac{\ell-1}2$, is odd.
\qed
\end{proof}

\noindent
An example of a graph satisfying assumptions of Theorem \ref{thm:unicyclic} is presented as the last graph in Figure \ref{slika}.

Finally, we also remark that if $G$ is a dual of a fullerene, then ${\gp}(G)$ can be an integer number or not an integer number, what is apparent from the results in \cite{ashrafi_koo_diu1}.

\section*{Acknowledgments}

The author Matev\v z \v Crepnjak acknowledges the financial support from the Slovenian Research Agency, research core funding  No. P1-0403.

The author Martin Knor acknowledges partial support by Slovak research
grants VEGA 1/0142/17, VEGA 1/0238/19, APVV-15-0220 and APVV-17-0428
and Slovenian research agency ARRS, program no.\ P1-0383.

The authors Niko Tratnik and Petra \v Zigert Pleter\v sek acknowledge the financial support from the Slovenian Research Agency, research core funding No.\ P1-0297 and J1-9109.

\end{document}